\documentclass[11pt]{article}
\usepackage{amscd, amsmath, amssymb, amsthm}
\usepackage[all,cmtip]{xy}
\usepackage[pagebackref]{hyperref}

\title{Hypersurfaces that are not stably rational}
\author{Burt Totaro}
\date{  }

\def\Z{\text{\bf Z}}
\def\Q{\text{\bf Q}}

\def\C{\text{\bf C}}
\def\P{\text{\bf P}}
\def\F{\text{\bf F}}

\def\arrow{\rightarrow}

\def\dR{\text{dR}}
\def\Hom{\text{Hom}}
\def\Spec{\text{Spec}}
\def\m{{\mathfrak m}}
\def\QQ{\overline{\Q}}

\begin{document}
\maketitle
\newtheorem{theorem}{Theorem}[section]
\newtheorem{corollary}[theorem]{Corollary}
\newtheorem{lemma}[theorem]{Lemma}

\theoremstyle{definition}
\newtheorem{definition}[theorem]{Definition}
\newtheorem{example}[theorem]{Example}

\theoremstyle{remark}
\newtheorem{remark}[theorem]{Remark}

A fundamental problem of algebraic geometry is to determine
which varieties are rational, that is, isomorphic to projective
space after removing lower-dimensional subvarieties from both
sides. In particular,
we want to know which smooth hypersurfaces
in projective space are rational. An easy case is that
smooth complex hypersurfaces of degree at least $n+2$ in $\P^{n+1}$
are not covered by rational curves and hence are not rational.

By far the most general result on rationality of hypersurfaces
is Koll\'ar's theorem that for $d$ at least $2\lceil (n+3)/3\rceil$,
a very general complex hypersurface of degree $d$
in $\P^{n+1}$ is not ruled and therefore
not rational \cite[Theorem 5.14]{Kollarbook}. Very little
is known about rationality in lower degrees, except
for cubic 3-folds and quintic 4-folds
\cite{CG}, \cite[Chapter 3]{Pukhlikov}.

A rational variety is also stably rational,
meaning that some product of the variety with projective space
is rational. Many techniques for proving non-rationality give
no information about stable rationality.
Voisin made a breakthrough in 2013 by showing
that a very general quartic double solid
(a double cover of $\P^3$ ramified over a quartic surface)
is not stably rational \cite{Voisin}. These
Fano 3-folds
were known to be non-rational over the complex numbers,
but stable rationality was an open question. Voisin's method
was to show that the Chow group of zero-cycles is not
universally trivial (that is, the Chow group becomes
nontrivial over some extension
of the base field), by degenerating the variety to a nodal 3-fold
which has a resolution of singularities $X$ with
nonzero torsion in $H^3(X,\Z)$.

Colliot-Th\'el\`ene and Pirutka simplified and generalized
Voisin's degeneration
method. They deduced
that very general quartic 3-folds are not stably rational \cite{CTP}.
This was striking, in that non-rationality of smooth quartic 3-folds
was the original triumph of Iskovskikh-Manin's work
on birational rigidity, while stable rationality
of these varieties was unknown \cite{IM}. Beauville applied the
method to prove that very general sextic double solids,
quartic double 4-folds, and quartic double 5-folds
are not stably rational \cite{Beauvillesextic, Beauvillequartic}.

In this paper, we show that a wide class of hypersurfaces 
in all dimensions are not
stably rational. Namely, for all $d\geq 2\lceil(n+2)/3\rceil$ and $n\geq 3$,
a very general complex hypersurface of degree $d$ in $\P^{n+1}$
is not stably rational (Theorem \ref{main}).
The theorem covers all the degrees in which Koll\'ar
proved non-rationality.
In fact, we get a bit more,
since Koll\'ar assumed $d\geq 2\lceil (n+3)/3\rceil$. For example,
very general quartic 4-folds are not stably rational,
whereas it was not even known whether these varieties are rational.

The method applies to some smooth hypersurfaces over $\QQ$
in each even degree. Section \ref{examples} gives
some examples over $\Q$ which
are not stably rational over $\C$.

The idea is that the most powerful results are obtained
by degenerating a smooth complex variety to a singular variety
in positive characteristic, rather than to a singular complex
variety. In fact, the best results arise by degenerating
to characteristic 2. We find
that very general hypersurfaces in the given degrees
have Chow group of zero-cycles not universally trivial,
which is stronger than being not stably rational.

Koll\'ar also proved non-rationality for several other classes
of rationally connected varieties, such as many ramified covers
of projective space or ramified covers of products of projective spaces
\cite{Kollarhyper, Kollarcover}, \cite[section V.5]{Kollarbook}.
The method of this paper should imply that those examples
are also not stably rational.

Using Theorem \ref{main}, Corollary \ref{rationality} shows for the first
time that a family of rational projective varieties
can specialize to a non-rational variety
with klt singularities. It remains unknown
whether rationality
specializes for smooth projective varieties, or for varieties
with canonical singularities.

I thank Hamid Ahmadinezhad, Jean-Louis Colliot-Th\'el\`ene,
Tommaso de Fernex,
J\'anos Koll\'ar, Adrian Langer,
and Claire Voisin for useful conversations.
This work was supported by The Ambrose Monell
Foundation and Friends, via the Institute for Advanced Study,
and by NSF grant DMS-1303105.

\section{Notation}

A property holds for {\it very general }complex points
of a complex variety $S$ if it holds for all points outside
a countable union of lower-dimensional closed subvarieties of $S$.
In particular, we can talk about properties of very general
hypersurfaces of a fixed degree $d$ in $\P^{n+1}_{\C}$,
since the space $S$ of all hypersurfaces of degree $d$ is an algebraic variety
(in fact, a projective space).

Let $R$ be a discrete valuation ring with fraction field $K$
and residue field $k$. Given a proper flat morphism
${\cal X}\arrow \Spec(R)$, we say that the general fiber
$X={\cal X}\times_R K$ {\it degenerates }to the special fiber
$Y={\cal X}\times_R k$. We also say that any base
change of $X$ to a larger field (perhaps algebraically closed)
degenerates to $Y$.

Let $X$ be a scheme of finite type over a field $k$. We say that
the Chow group of zero-cycles of $X$ is {\it universally trivial }if
the flat pullback homomorphism $CH_0(X)\arrow CH_0(X_E)$
is surjective for every field $E$ containing $k$. For $X$
a smooth proper variety over $k$, universal triviality
of $CH_0$ is equivalent to many other conditions:
the degree map $CH_0(X_E)\arrow \Z$ is an isomorphism
for every field $E$ containing $k$,
or $X$ has a decomposition of the diagonal
of a certain type (written out in the proof of Lemma \ref{form}),
or $X$ has trivial unramified cohomology
with coefficients in any cycle module \cite{Merkurjev},
or all Chow groups of $X$ below the top dimension
are universally supported on a divisor. A reference
for these equivalences is \cite[Theorem 2.1]{Totaromotive}.

We use the following fact, due to Colliot-Th\'el\`ene and Coray
in characteristic zero and to Fulton in general \cite[Proposition 6.3]{CC},
\cite[Example 16.1.12]{Fulton}. Fulton assumes that the base field
is algebraically closed, but the proof works without that. These
references only treat
Theorem \ref{zero} for birational equivalence, but it follows for stable
birational equivalence by the formula for the Chow groups
of $X\times \P^n$ \cite[Theorem 3.3]{Fulton}.

\begin{theorem}
\label{zero}
The Chow group of zero-cycles is invariant under stable birational
equivalence for smooth projective varieties over a field.
\end{theorem}

\section{Hypersurfaces}

\begin{theorem}
\label{main}
Let $X$ be a very general hypersurface of degree $d$ in
$\P^{n+1}_{\C}$ with $n\geq 3$. If $d\geq 2\lceil (n+2)/3\rceil$,
then $CH_0$ of $X$ is not universally trivial. It follows
that $X$ is not stably rational. For $d$ even, the conclusions also hold
for some smooth hypersurfaces over $\QQ$.
\end{theorem}

\begin{proof}

We first prove this when the degree $d$ is even, $d=2a$.
Then a smooth hypersurface of degree $2a$ can degenerate to
a double cover of a hypersurface of degree $a$.
We consider such a degeneration to an inseparable double
cover in characteristic 2. 

Explicitly, following Mori \cite[Example 4.3]{Mori}, let $R$ be a discrete
valuation ring, and let $S$ be the weighted projective space
$\P(1^{n+2}a)=\P(x_0,\ldots,x_{n+1},y)$ over $R$.
Let $f,g\in R[x_0,\ldots,x_{n+1}]$ be homogeneous polynomials
of degree $2a$ and $a$, respectively.
Let $t$ be a uniformizer for $R$.
Let $Z$ be the complete intersection subscheme of $S$ defined by $y^2=f$
and $g=ty$.
Then the generic fiber
of $Z$ over $R$ (where $t\neq 0$) is isomorphic to the hypersurface
$g^2-t^2f=0$ in $\P^{n+1}$ of degree $2a$,
whereas the special fiber (where $t=0$)
is a double cover of the hypersurface $g=0$ in $\P^{n+1}$ of degree $a$.
We consider the case where the residue field $k$ of $R$ has characteristic
2; then the special fiber $Y$ is an inseparable double cover of $\{ g=0\}$.

Assume that the residue field $k$ is algebraically closed.
Using only that $a\geq 2$,
Koll\'ar showed that for general polynomials $f$ and $g$,
the singularities
of $Y$ are etale-locally isomorphic
to $0=y^2+x_1x_2+x_3x_4+\cdots+x_{n-1}x_n+f_3$ if $n$ is even,
or to $0=y^2+x_1^3+x_2x_3+x_4x_5+\cdots+x_{n-1}x_n+f_3$ if $n$ is odd
\cite[proof of Theorem V.5.11]{Kollarbook}. Here
$f_3\in (x_1,\ldots,x_n)^3$ and, for $n$ odd,
the coefficient of $x_1^3$ in $f_3$ is zero.

Then one computes that simply blowing up the singular points gives
a resolution of singularities $Y'\arrow Y$. Moreover, each exceptional
divisor of this resolution is isomorphic to a quadric $Q^{n-1}$
over $k$, which is smooth if $n$ is even and is singular at one point
if $n$ is odd. This quadric is a smooth projective rational variety over $k$
for $n$ even, and it is a projective cone for $n$ odd.
By Theorem \ref{zero} (for $n$ even), the Chow group $CH_0$
of each exceptional divisor of $Y'\arrow Y$
is universally trivial.
Therefore, for every extension field $E$ of $k$, the pushforward
homomorphism $CH_0Y'_E\arrow CH_0Y_E$ is an isomorphism.

For a smooth $n$-fold $X$ over a field,
the {\it canonical bundle }is the line bundle $K_X=\Omega^n_X$.
Let $M$ be the pullback to $Y'$ of the line bundle
$K_{\{ g=0\} }\otimes O(a)^{\otimes 2}\cong O(-n-2+3a)$ on
the hypersurface $\{ g=0\}\subset \P^{n+1}$. Since
$a\geq \lceil (n+2)/3\rceil$, we have $H^0(Y',M)\neq 0$.
Next, for $n\geq 3$,
Koll\'ar computed that there is a nonzero map from the
line bundle $M$ to $\Omega^{n-1}_{Y'}$ \cite[proof
of Theorem V.5.11]{Kollarbook}.
In particular, it follows that the $n$-fold $Y'$
has $H^0(Y',\Omega^{n-1})\neq 0$.

Under the slightly stronger assumption that $a\geq \lceil
(n+3)/3\rceil$, the line bundle $M$ is big, which Koll\'ar
used to show that $Y$ is not separably uniruled. This is a strong
conclusion. It follows, for example,
that there is no generically smooth dominant
rational map from a separably rationally connected variety
to $Y$. In particular, $Y$ is not stably rational. With that
approach, however, it was not clear how to show
that a lift of $Y$ to characteristic zero is not stably rational.

\begin{lemma}
\label{form}
Let $X$ be a smooth projective variety over a field $k$.
If $H^0(X,\Omega^i)$ is not zero for some $i>0$,
then $CH_0$ of $X$ is not universally trivial. More precisely,
if $k$ has characteristic zero, then $CH_0\otimes\Q$ is not universally
trivial; and if $k$ has characteristic $p>0$, then
$CH_0/p$ is not universally trivial.
\end{lemma}

\begin{proof}
This is part of the ``generalized Mumford theorem''
for $k$ of characteristic zero.
The proof by Bloch and Srinivas \cite[Theorem 3.13]{Voisinbook},
using the cycle class map in de Rham cohomology, is
similar to what follows.

For $k$ of characteristic $p>0$,
the hypothesis that $H^0(X,\Omega^i)\neq 0$ remains true
after enlarging $k$, and so we can assume that $k$ is perfect.
In that case, we can use the cycle class map
in de Rham cohomology constructed by Gros \cite[section II.4]{Gros}. Namely,
for a smooth scheme $X$ over $k$, let $\Omega^i_{\log}$ be the subsheaf
of $\Omega^i$ on $X$ in the etale topology generated
by products $df_1/f_1\wedge\cdots \wedge df_i/f_i$ for nonvanishing
regular functions
$f_1,\ldots,f_i$. This is a sheaf of $\F_p$-vector spaces,
not of $O_X$-modules.
Gros defined a cycle map
from $CH^i(X)/p$ to $H^i(X,\Omega^i_{\log})$, which maps both to
$H^i(X,\Omega^i)$ and to de Rham cohomology $H^{2i}_{\dR}(X/k)$.

As a result, for $X$ smooth over $k$ and $Y$ smooth proper over $k$, with
$y=\dim(Y)$,
a correspondence $\alpha$ in $CH^y(X\times_k Y)/p$
determines a pullback map $\alpha^*\colon H^0(Y,\Omega^i)
\arrow H^0(X,\Omega^i)$ for all $i$. Explicitly, $\alpha$
has a class in 
\begin{align*}
H^y(X\times Y,\Omega^y_{X\times Y}) &= H^y(X\times Y,
\oplus_j \Omega^j_X\otimes \Omega^{y-j}_Y)\\
&= \oplus_{i,j} H^i(X,\Omega^j)\otimes H^{y-i}(Y,\Omega^{y-j}).
\end{align*}
In particular, $\alpha$ determines an element $\alpha^*$ of $H^0(X,\Omega^i)
\otimes H^{y}(Y,\Omega^{y-i})$, which is $\Hom(H^0(Y,\Omega^i),
H^0(X,\Omega^i))$ by Serre duality, as we want. (The last step uses
properness of $Y$ over $k$.)

Let $X$ be a smooth projective variety over $k$ with $CH_0$
universally trivial, or just with $CH_0/p$ universally trivial. Since the
diagonal $\Delta$ in $CH^n(X\times X)$ restricts (over the generic
point of the first copy of $X$) to a zero-cycle of degree 1
on $X_{k(X)}$, our assumption implies that there
is a zero-cycle $x$ on $X$ such that $\Delta=x$ in $CH_0(X_{k(X)})/p$.
Equivalently, there is a decomposition of the diagonal,
$$\Delta=X\times x+B$$
in $CH^n(X\times X)/p$ for some cycle $B$ supported on $S\times X$
with $S$ a closed subset not equal to $X$. For any $i>0$,
the pullback $\Delta^*$ from the second copy
of $H^0(X,\Omega^i)$ to the first
is the identity, but the pullback by $X\times x$
or by $B$ is zero. (For $B$, this uses that the restriction of $B$
to $(X-S)\times X$ is zero, and hence
its class in $H^n((X-S)\times X,\Omega^n)$
is zero. As a result, for any form $\theta$ in $H^0(X,\Omega^i)$,
the restriction of $B^*\theta$ to $H^0(X-S,\Omega^i)$ is
zero. But $H^0(X,\Omega^i)\arrow H^0(X-S,\Omega^i)$ is injective,
and so $B^*\theta=0$.)
It follows that $H^0(X,\Omega^i)=0$.
\end{proof}

By Lemma \ref{form},
the resolution $Y'$ discussed above has $CH_0$
not universally trivial. We also know that the resolution $Y'\arrow Y$ induces
an isomorphism on $CH_0$ over all extension fields of $k$.
It follows that any variety $X$
over an algebraically closed field of characteristic zero
that degenerates to $Y$ has $CH_0$ not universally trivial,
by the following result, due in this form to Colliot-Th\'el\`ene
and Pirutka \cite[Th\'eor\`eme 1.12]{CTP}.

\begin{theorem}
\label{ctp}
Let $A$ be a discrete valuation ring with fraction field $K$
and residue field $k$, with $k$ algebraically closed.
Let ${\cal X}$ be a flat proper scheme over $A$
with geometrically integral fibers.
Let $X$ be the general fiber
${\cal X}\times_A K$ and $Y$ the special fiber ${\cal X}\times_A k$.
Assume that there is a proper birational morphism $Y'\arrow Y$
with $Y'$ smooth over $k$
such that $CH_0Y'\arrow CH_0Y$ is universally an isomorphism.
Let $\overline{K}$ be
an algebraic closure of $K$.
Assume that there is a proper
birational morphism $X'\arrow X$ with $X'$ smooth over $K$
such that $CH_0$ of $X'_{\overline{K}}$ is universally trivial.
Then $Y'$ has $CH_0$ universally trivial.
\end{theorem}

As a result,
a very general complex hypersurface of degree $d=2a$ (where
$a\geq \lceil (n+2)/3\rceil$) has $CH_0$ not universally trivial.
In particular, it is not stably rational, by Theorem \ref{zero}.

It remains to consider hypersurfaces of degree $2a+1$, for
$a\geq \lceil (n+2)/3\rceil$. The following approach seems clumsy,
but it works. 

If $2a+1>n+1$, then every smooth hypersurface $W$ of degree $2a+1$
in $\P^{n+1}_{\C}$
has $H^0(W,K_W)\neq 0$, and so $CH_0W$ is not universally
trivial by Lemma \ref{form}. So we can assume that $2a+1\leq n+1$.

Observe that a very general hypersurface of degree $2a+1$
in $\P^{n+1}_{\C}$
degenerates to the union of a hypersurface $X$ of degree $2a$
in $\P^{n+1}_{\C}$ and
a hyperplane $H$. Since the special fiber of that degeneration is reducible,
we need the following variant of Theorem \ref{ctp}, which I worked out
with Colliot-Th\'el\`ene:

\begin{lemma}
\label{reducible}
Let $A$ be a discrete valuation ring with fraction field $K$
and algebraically closed residue field $k$.
Let ${\cal X}$ be a flat proper scheme
over $A$. Let $X$ be the general fiber
${\cal X}\times_A K$ and $Y$ the special fiber ${\cal X}\times_A k$.
Suppose that $X$ is geometrically
integral and there is a proper birational morphism
$X'\arrow X$ with $X'$ smooth over $K$. Suppose
that there is an algebraically closed field $F$ containing $K$
such that $CH_0$ of $X'_F$ is universally trivial.
Then, for every extension field $l$ of $k$,
every zero-cycle of degree zero in the smooth locus
of $Y_l$ is zero in $CH_0(Y_l)$.
\end{lemma}

\begin{proof}
After replacing $A$ by its completion, we can assume that $A$ is complete.
Then there is an algebraically closed field $F$ containing
the fraction field $K$ of the new ring $A$
such that $CH_0$ of $X'_F$ is universally trivial.
So the class of the diagonal in $CH_0(X'_{F(X')})$
is equal to the class of a $\overline{K}$-point of $X'$. By a specialization
argument, it follows that
there is a finite extension $L$ of $K$ such that the class of the diagonal
in $CH_0(X'_{L(X')})$ is equal to the class of an $L$-point
of $X'$. Since $A$ is complete, the integral closure of $A$ in $L$
is a complete discrete valuation ring
\cite[Proposition II.2.3]{Serre}.
Its residue field is $k$. We now replace $A$ by this discrete valuation
ring and ${\cal X}$ by its pullback to that ring, without changing
the special fiber $Y$.

For any field extension $l$ of $k$, there is a flat local $A$-algebra $B$
such that $\m_B=\m_AB$ and $B$ is a discrete valuation
ring with residue field $l$,
by ``inflation of local rings'' \cite[Ch.~IX, Appendice, Corollaire
du Th\'eor\`eme 1]{Bourbaki}. Let $C$ be the completion of $B$.
Let $M$ be the fraction field of $C$, which is an extension of the field $L$
above. The residue field of $C$
is $l$.
Consider the $C$-scheme ${\cal X}\times_A C$. Its generic fiber
is $X_M$, which has the resolution $X'_M$.
Its special fiber is $Y_l$. 
Since the degree map $CH_0(X'_M)\arrow \Z$ is an isomorphism,
every zero-cycle of degree zero in the smooth locus of $Y_l$
is zero in $CH_0Y_l$ \cite[Proposition 1.9]{CTP}.
\end{proof}

Since a very general hypersurface $W$ of degree $2a+1$
in $\P^{n+1}_{\C}$
degenerates to the union of a very general hypersurface $X$ of degree $2a$
in $\P^{n+1}_{\C}$ and
a very general hyperplane $H$, Lemma \ref{reducible} shows that $CH_0$
of $W$ is not universally trivial
if there is a zero-cycle of degree zero on $X_{k(X)}-(X\cap H)_{k(X)}$
which is not zero in $CH_0(X\cup H)_{k(X)}$.
This holds if we can show that $CH_0(X\cap H)_{k(X)}\arrow CH_0X_{k(X)}$
is not surjective.

If that homomorphism is surjective, then
we have a decomposition of the diagonal
$$\Delta=A+B$$
in $CH_n(X\times X)$, where $A$ is supported on $X\times (X\cap H)$
and $B$ is supported on $S\times X$ for some closed subset $S$ not
equal to $X$. Let $Y$ be the singular variety
in characteristic 2 to which $X$ degenerates (as in the earlier
part of the proof). By the specialization homomorphism on Chow
groups \cite[Proposition 2.6, Example 20.3.5]{Fulton}, the decomposition
of the diagonal in $X\times X$ gives
a similar decomposition
of the diagonal in $Y\times Y$. That is, the class of the diagonal
in $CH_0Y_{k(Y)}$ is in the image of $CH_0(Y_H)_{k(Y)}$,
where $Y_H$ is the subvariety of $Y$ to which
$X\cap H$ degenerates. (Here $Y_H$ is an inseparable double
cover of a hypersurface, like $Y$ itself. We can arrange for
$Y_H$ to be disjoint from the singular set of $Y$, but $Y_H$
will still be singular, with finitely many singular points
of the form described above for $Y$, one dimension lower.)
Since the resolution $Y'$ of $Y$ we constructed has $CH_0Y'\arrow CH_0Y$
universally an isomorphism, it follows that the class
of the diagonal in $CH_0Y'_{k(Y')}$ is in the image
of $CH_0(Y'_H)_{k(Y')}$, where the inverse image $Y'_H$
of $Y_H$ is isomorphic
to $Y_H$. That determines a decomposition of the diagonal
$$\Delta=A+B$$
in $CH_n(Y'\times Y')$, where $A$ is supported on $Y'\times Y'_H$
and $B$ is supported on $S\times Y'$
for some closed subset $S$ not equal to $Y'$.

Now consider the action of correspondences (by pullback
from the second factor to the first) on $H^0(Y',\Omega^{i})$,
for any $i>0$. 
Here $Y'_H$ is singular, with only singularities of the form
described above for $Y$, one dimension lower.
Let $Z'$ be the resolution of $Y'_H$ obtained by blowing up
the singular points.
The diagonal $\Delta$ acts as the identity on $H^0(Y',\Omega^i)$,
and $B$ acts by $0$.
For any form $\beta $ in $H^0(Y',\Omega^{i})$ that pulls
back to zero on $Z'$, $A$ acts by $0$ on $\beta$,
and so the decomposition above gives that $\beta=0$.
That is, we have shown that the restriction
$$H^0(Y',\Omega^{i})\arrow H^0(Z',\Omega^i)$$
is injective for $i>0$.

By viewing $Y$ as a complete intersection in the smooth
locus of a weighted projective space,
we see that $K_Y$ is isomorphic to $O(-n-2+2a)$,
and likewise $K_{Y_H}$ is isomorphic to $O(-n-1+2a)$.
Since we arranged that $2a+1\leq n+1$, $Y_H$ is Fano,
and in particular $H^0(Y_H,K_{Y_H})=0$.
It follows
that the resolution $Z'$ of $Y_H$ has
$H^0(Z',\Omega^{n-1})=H^0(Z',K_{Z'})=0$. 
So the previous paragraph's injection gives that $H^0(Y',\Omega^{n-1})=0$.
Since $a\geq \lceil (n+2)/3\rceil$,
this contradicts the calculation that $H^0(Y',\Omega^{n-1})\neq 0$,
as discussed before
Lemma \ref{form}. We conclude that in fact $CH_0$
of a very general hypersurface of degree $2a+1$ in $\P^{n+1}_{\C}$
is not universally trivial.

Finally, in each even degree covered by this theorem,
there are non-stably-rational smooth hypersurfaces
over $\QQ$. Any hypersurface that reduces to a suitable
double cover over $\overline{\F_2}$ is not stably rational,
and most such hypersurfaces
are smooth. We do not immediately get examples
of odd degree over $\QQ$, since the argument involves two
successive degenerations: first, degenerate a hypersurface
of degree $2a+1$ to a hypersurface of degree $2a$ plus a hyperplane,
and then degenerate the hypersurface of degree $2a$ to a double cover
in characteristic 2.
\end{proof}

\section{Examples over the rational numbers}
\label{examples}

In each even degree covered by Theorem \ref{main},
there are non-stably-rational smooth hypersurfaces
over $\QQ$. We can use
any hypersurface that reduces to a suitable
double cover over $\overline{\F_2}$.
In fact, if a double cover
with only singularities as in the proof of Theorem \ref{main}
exists over $\F_2$,
then there are smooth hypersurfaces over $\Q$ which are not stably
rational over $\C$. We now use this method
to give examples
of smooth quartic 3-folds and smooth quartic 4-folds
over $\Q$ that are
not stably rational over $\C$.

Joel Rosenberg gave some elegant examples of inseparable double covers
of projective space over $\F_2$ with only singularities
as above \cite[section 4.7]{KSC}.
Namely, for any even $n$ and even $d$, the double
cover of $\P^n$ over $\F_2$ ramified over the hypersurface
$x_0^{d-1}x_1+\cdots+x_{n-1}^{d-1}x_n+x_n^{d-1}x_0=0$
has the desired singularities.
It would be interesting to find equally simple examples
of inseparable double covers $Y$ of smooth hypersurfaces
over $\F_2$ such that $Y$ has singularities as above; that is the problem
we encounter here.

\begin{example}
Quartic 4-folds.
\end{example}

For quartic 4-folds, even non-rationality was previously unknown.
The free program Macaulay2 shows that the 4-fold $\{ y^2=f, g=0\}$
in $\P(1^62)=\P(x_0,\ldots,x_5,y)$ over $\F_2$ has only
singularities as in the proof of Theorem \ref{main}
in the example
\begin{align*}
f &= x_0^3x_2+x_0x_3^3+x_1^3x_4+x_1^2x_2x_4+x_1x_4^3+x_2^3x_5
+x_3x_5^3\\
g &= x_0x_1+x_2x_3+x_4x_5.
\end{align*}

The proof of Theorem \ref{main} shows
that any hypersurface $X$ in $\P^5_{\Q}$ of the form $g^2-4f=0$,
where $g$ and $f$ are lifts to $\Z$ of the polynomials above,
has $CH_0$ of $X_{\C}$ not universally trivial. Therefore,
$X_{\C}$ is not stably rational. For example, this applies
to the following quartic 4-fold, which we compute is smooth:
\begin{multline*}
0=x_0^2x_1^2-4x_0^3x_2+2x_0x_1x_2x_3+x_2^2x_3^2
-4x_0x_3^3-4x_1^3x_4\\
-4x_1^2x_2x_4-4x_1x_4^3-4x_2^3x_5
+2x_0x_1x_4x_5+2x_2x_3x_4x_5+x_4^2x_5^2-4x_3x_5^3.
\end{multline*}

\begin{example}
Quartic 3-folds.
\end{example}

Here is an example of a smooth quartic 3-fold
over $\Q$ that is not stably rational over $\C$.
(Following a suggestion by Wittenberg,
Colliot-Th\'el\`ene and Pirutka
gave examples over $\QQ$ \cite[Appendice A]{CTP}.)
Macaulay2 shows that the 3-fold $\{ y^2=f, g=0\}$
in $\P(1^52)=\P(x_0,\ldots,x_4,y)$ over $\F_2$ has only
singularities as in the proof of Theorem \ref{main}
in the example
\begin{align*}
f &= x_0^3x_2+x_1^3x_2+x_0x_1^2x_4+x_0x_2^2x_4+x_3^3x_4\\
g &= x_0x_1+x_2x_3+x_4^2.
\end{align*}

As a result, any hypersurface $X$ in $\P^4_{\Q}$ of the form $g^2-4f=0$,
where $g$ and $f$ are lifts to $\Z$ of the polynomials above,
has $CH_0$ of $X_{\C}$ not universally trivial. Therefore,
$X_{\C}$ is not stably rational. For example, this applies
to the following quartic 3-fold, which we compute is smooth:
\begin{multline*}
0=x_0^2x_1^2-4x_0^3x_2-4x_1^3x_2-8x_2^4+2x_0x_1x_2x_3
+x_2^2x_3^2\\
-4x_0x_1^2x_4-4x_0x_2^2x_4-4x_3^3x_4+2x_0x_1x_4^2
+2x_2x_3x_4^2+x_4^4.
\end{multline*}

\section{Rationality in families}

Given a family of projective varieties for which
the geometric generic fiber is rational,
are all fibers geometrically
rational? The analogous question for geometric ruledness
has a positive answer, by Matsusaka
\cite[Theorem IV.1.6]{Kollarbook}. The statement for geometric rationality
is easily seen to be false if we make no restriction
on the singularities; for example, a smooth cubic surface $X$ can degenerate
to the cone $Y$ over a smooth cubic curve. Here $X$ is rational
over $\C$, but $Y$ is not. In this example, $Y$ is log canonical
but not klt (Kawamata log terminal).

The following application of Theorem \ref{main}, suggested
by Tommaso de Fernex, shows that rationality need not specialize
even when the singularities allowed are mild, namely klt.
It remains an open question whether rationality specializes
when all fibers are smooth, or at least canonical.
If the dimension is at most 3
and the characteristic is zero, rationality
specializes when all fibers are klt, by
de Fernex and Fusi \cite[Theorem 1.3]{DFF} and Hacon
and M\textsuperscript{c}Kernan
\cite[Corollary 1.5]{HM}.

\begin{corollary}
\label{rationality}
For any $m\geq 4$, there is a flat family of projective $m$-folds
over the complex affine line
such that all fibers over $A^1-0$
are smooth rational varieties,
while the fiber over 0 is klt and not rational.
\end{corollary}

\begin{proof}
For a variety $X$ embedded in projective space $\P^N$ (with a proviso
to follow), $X$ degenerates
to the projective cone over a hyperplane section $X\cap H$.
To see this, consider the projective cone
over $X$ in $\P^{N+1}$, and intersect it with a pencil $\P^1$
of hyperplanes in $\P^{N+1}$. For most hyperplanes in the pencil,
the intersection is isomorphic to $X$, while for a hyperplane
through the node, the underlying set of the intersection
is the projective cone over $X\cap H$. On the other hand,
that intersection could be different from the projective cone
over $X\cap H$ as a scheme because it is non-reduced at the cone point.
That problem does not arise if $X\subset \P^N$ is projectively normal
and $H^1(X,O(j))=0$ for all $j\geq -1$ \cite[Proposition 3.10]{Kollarsing},
as will be true in the following example.

By considering a Veronese embedding of a projective space $X=\P^m$,
it follows that for any
hypersurface $W$ of degree $d$ in $\P^m$, $\P^m$ degenerates
to the projective cone $Y$ over $(W,O(d))$.
If the hypersurface $W$ is not stably rational, then this is a degeneration
of a smooth projective rational variety (in fact, projective 
space) to a variety $Y$ which is not rational. Indeed, $Y$ is birational
to $\P^1\times W$. (So just knowing that $W$ is not rational
would not be enough.)

It remains to describe the singularities of $Y$. Namely, the projective cone
$Y$ over $(W,O(d))$ is klt if and only if $d\leq m$, meaning that
$W$ is Fano \cite[Lemma 3.1]{Kollarsing}. (Note that $Y$ is {\it not }the
projective cone over $(W,O(1))$ in $\P^{m+1}$, which has a milder
singularity.)
For every $m\geq 4$, Theorem \ref{main} gives
a smooth Fano hypersurface $W$ in $\P^m$ which is not stably rational.
(For $m=4$, this is Colliot-Th\'el\`ene and Pirutka's theorem
\cite{CTP}.)
Thus we have a degeneration of a smooth projective rational $m$-fold to a klt
variety which is not rational.
\end{proof}


\small \sc UCLA Mathematics Department, Box 951555,
Los Angeles, CA 90095-1555

totaro@math.ucla.edu
\end{document}